\documentclass[12pt]{amsart}

\usepackage{amsmath}
\usepackage{amssymb}
\usepackage{amsfonts}
\usepackage{amsthm}
\usepackage{enumerate}
\usepackage{color}
\usepackage{psfrag}
\usepackage{mathrsfs}
\usepackage{amscd}
\usepackage[hyperindex,backref=page]{hyperref}
\usepackage{graphicx,color}
\usepackage{tikz} 

\usetikzlibrary{positioning, intersections, calc, arrows.meta,math}
\def\NZQ{\mathbb}
\def\NN{{\NZQ N}} % the font for N,Z,Q,R,C

\def\ZZ{{\NZQ Z}}
\def\RR{{\NZQ R}}

\def\opn#1#2{\def#1{\operatorname{#2}}} % to make operators
\opn\chara{char} \opn\length{\ell} \opn\pd{pd} \opn\rk{rk}
\opn\projdim{proj\,dim} \opn\injdim{inj\,dim} \opn\rank{rank}
\opn\depth{depth} \opn\codepth{codepth} \opn\grade{grade}
\opn\height{height} \opn\embdim{emb\,dim} \opn\codim{codim}

\opn\Tr{Tr} \opn\bigrank{big\,rank}
\opn\superheight{superheight}\opn\lcm{lcm}
\opn\trdeg{tr\,deg}%
\opn\reg{reg} \opn\lreg{lreg} \opn\skel{skel} \opn\Gr{Gr}
\opn\dim{dim} \opn\indeg{indeg} \opn\Ass{Ass} \opn\Min{Min}
\opn\div{div} \opn\Div{Div} \opn\cl{cl} \opn\Cl{Cl}
\opn\Spec{Spec} \opn\Supp{Supp} \opn\supp{supp} \opn\Sing{Sing}
\opn\Ass{Ass}
\opn\Ann{Ann} \opn\Rad{Rad} \opn\Soc{Soc}
\opn\Sym{Sym} \opn\Ker{Ker} \opn\Coker{Coker} \opn\Im{Im}
\opn\Hom{Hom} \opn\Tor{Tor} \opn\Ext{Ext} \opn\End{End}
\opn\Aut{Aut} \opn\id{id} \opn\ini{in} \opn\tr{tr}

\def\core{core}
\opn\nat{nat}\opn\it{it}
\opn\pff{proof}%   \pf exists already
\opn\Pf{proof} \opn\GL{GL} \opn\SL{SL} \opn\mod{mod} \opn\ord{ord}
\opn\diam{diam}
\opn\dist{dist}
\opn\aff{aff} \opn\con{conv} \opn\relint{relint} \opn\st{st}
\opn\lk{lk} \opn\cn{cn} \opn\core{core} \opn\vol{vol}
\opn\link{link} \opn\star{star} \opn\skel{skel}
\opn\gr{gr}

\def\pot#1#2{#1[\kern-0.28ex[#2]\kern-0.28ex]}

\opn\dirlim{\underrightarrow{\lim}}
\opn\inivlim{\underleftarrow{\lim}}
\def\Implies{\ifmmode\Longrightarrow \else
     \unskip${}\Longrightarrow{}$\ignorespaces\fi}
\def\implies{\ifmmode\Rightarrow \else
     \unskip${}\Rightarrow{}$\ignorespaces\fi}
\def\iff{\ifmmode\Longleftrightarrow \else
     \unskip${}\Longleftrightarrow{}$\ignorespaces\fi}

\let\:=\colon
\theoremstyle{plain}
\newtheorem{thm}{Theorem}[section]
\newtheorem{lemma}[thm]{Lemma}
\newtheorem{prop}[thm]{Proposition}
\newtheorem{cor}[thm]{Corollary}

\newtheorem{quest}[thm]{Question}

\newtheorem*{thm-q}{Theorem}
\newtheorem*{cor-q}{Corollary}
\newtheorem*{quest-q}{Question}
\newtheorem*{quests-q}{Questions}
\theoremstyle{definition}

\newtheorem{exam}[thm]{Example}

\newtheorem*{acknowledgement}{Ackowledgement}
\theoremstyle{remark}

\let\epsilon\varepsilon
\let\phi=\varphi
\let\kappa=\varkappa
\opn\Gin{Gin}

\opn\inii{in} \opn\inim{inm} \opn\rate{rate}
\numberwithin{equation}{section}

\begin{document}
\title[The v-numbers of Stanley--Reisner ideals]{The v-numbers of Stanley--Reisner ideals from the viewpoint of Alexander dual complexes}
%%%%%%%%%%%%%%%%%%%%%%%%%%%%%%%%%%%%%%%%%%%%%%%%%%%%%%%%%%%%%%%%%%%%%%
%%Information for first author
%%%%%%%%%%%%%%%%%%%%%%%%%%%%%%%%%%%%%%%%%%%%%%%%%%%%%%%%%%%%%%%%%%%%%%
\author[T. Kataoka]{Tatsuya Kataoka}
\address{(T. Kataoka) Department of Mathematics, Okayama University, 3-1-1 Tsushima-naka, Kita-ku, Okayama 700-8530, Japan.}
\email{pr2n839d@s.okayama-u.ac.jp}
%%%%%%%%%%%%%%%%%%%%%%%%%%%%%%%%%%%%%%%%%%%%%%%%%%%%%%%%%%%%%%%%%%%%%%
%% Information for second author
%%%%%%%%%%%%%%%%%%%%%%%%%%%%%%%%%%%%%%%%%%%%%%%%%%%%%%%%%%%%%%%%%%%%%%
\author[Y. Muta]{Yuji Muta}
\address{(Y. Muta) Department of Mathematics, Okayama University, 3-1-1 Tsushima-naka, Kita-ku, Okayama 700-8530, Japan.}
\email{p8w80ole@s.okayama-u.ac.jp}
%%%%%%%%%%%%%%%%%%%%%%%%%%%%%%%%%%%%%%%%%%%%%%%%%%%%%%%%%%%%%%%%%%%%%%
%% Information for third author
%%%%%%%%%%%%%%%%%%%%%%%%%%%%%%%%%%%%%%%%%%%%%%%%%%%%%%%%%%%%%%%%%%%%%%
\author[N. Terai]{Naoki Terai}
\address{(N. Terai) Department of Mathematics, Okayama University, 3-1-1 Tsushima-naka, Kita-ku, Okayama 700-8530, Japan.}
\email{terai@okayama-u.ac.jp}
%%%%%%%%%%%%%%%%%%%%%%%%%%%%%%%%%%%%%%%%%%%%%%%%%%%%%%%%%%%%%%%%%%%%%%
%    General info
%%%%%%%%%%%%%%%%%%%%%%%%%%%%%%%%%%%%%%%%%%%%%%%%%%%%%%%%%%%%%%%%%%%%%%
 %%% 13D02 Homological methods, Syzygies and resolutions
 %%% 13F55 Face and Stanley--Reisner rings; simplicial complexes 
 %%% 13H10 Special types (Cohen-Macaulay, Gorenstein, Buchsbaum, etc.)
 %%% 05E99 Algebraic combinatorics, None of the above, but in this section
 %%% 13D45 Homological methods, Local cohomology
%\date{\today}
\dedicatory {}

\begin{abstract}
 We express the v-number of the Stanley--Reisner ideal   in terms of its Alexander dual complex
and prove that  the v-number of a cover ideal  is just two less than the initial degree of its  syzygy module. We give   some relation between the  v-number of the Stanley--Reisner ideal  and  the Serre- depth of the  quotient ring of the second symbolic power of the Stanley--Reisner ideal of   its Alexander dual. We also show that  the v-number of  the Stanley--Reisner ideal of a 2-pure  simplicial complex is equal to the dimension of its  Stanley--Reisner ring.
\end{abstract}

\maketitle

%%%%%%%%%%%%%%%%%%%%%%%%%%%%%%%%%%%%%%%%%%%%%%%%%%%%%%%%%%%%%%%%%%%%%%
\section{Introduction}

Let $k$ be a fixed field, $S=k[x_1,\ldots,x_n]$ be a polynomial ring with $\deg x_i= 1 $ for all $i \in [n]=\{1,2,\ldots,n\}$ and let $I$ be a homogeneous ideal of $S$.
The \textit{$\mbox{\rm v}$-number} of $I$ with respect to an associated prime $P$ of $I$ is defined to be $$\mbox{\rm v}_P( I)=\min \{ \deg f \,:\, (I:f) =P\}.$$We set $$\mbox{\rm v}( I)=\min \{ \mbox{\rm v}_P( I) \,:\, P \in \Ass S/I \}$$and$$\mbox{\rm v}_i( I)=\min \{ \mbox{\rm v}_P( I) \,:\, P \in \Ass S/I, {\rm ht} P=i\}$$for $i\ge 0$, where $\mathrm{Ass}(S/I)$ denotes the set of associated prime ideals of $I$ in $S$. For convenience, we set $\mbox{\rm v}_i( I)=\infty$, if there is no associated prime of $I$ with height $i$.

The  v-number was first introduced in the context with coding theory in \cite{CSTPV}.
In the field of combinatorial commutative algebra, the case that $I$ is a squarefree monomial ideal is mainly treated and the relation with the Castelnuovo--Mumford regularity has been considered. See \cite{C, JV, SS, S}. In this paper, we consider the v-number of a Stanley--Reisner ring  from the viewpoint of the Alexander dual complex.

After we introduce necessary notation in Section 2, we discuss Takayama's formula  \cite[Theorem 1]{Ta} in Section 3. Takayama's formula  is used to analyze the local cohomology modules of  the residue ring by a monomial ideal. To investigate the local cohomology modules of the residue rings by the symbolic powers of Stanley--Reisner ideals, we reformulate Takayama's formula (Theorem \ref{Takayama}). 

In Section 4, we express the v-number of the Stanley--Reisner ideal of a simplicial complex  in terms of its Alexander dual complex: 
\begin{thm}[\mbox{\rm see Theorem \ref{intersection}}]
Let $\Delta$ be a simplicial complex on the vertex set $ V=\{x_1, \dots, x_n\}$. Assume $P=(x_1, x_2, \dots,  x_h)$ be an associated prime of  $I_{\Delta}$. Then $$\mbox{\rm v}_P( I_{\Delta})=n-h -\max \{ |\cap_{j=1}^{h} F_j | \,:\, F_j \in \mathcal{F}(\Delta^{*}) \mbox{ such that }V|_{[h]\setminus{\{i\}}}\subset F_j \mbox{ for all } j \},$$where $V|_{[h]\setminus{\{i\}}}=\{x_{1},\ldots, \hat{x_{i}}, \ldots, x_{h}\}$ and the hat $\hat{x}_i$ indicates that the element $x_i$ is removed.
\end{thm}
This result leads us to a useful formula for cover ideals of finite simple graphs (Corollary \ref{cover ideals}). 

In Section 5, we consider {\rm v}-numbers of Stanley--Reisner ideals with pure height 2 by associating the {\rm v}-number with graded Betti numbers of the minimal free resolution of Stanley--Reisner ideals. In particular, we prove the following:
\begin{thm}[\mbox{\rm see Corollary \ref{pureheight2}}]
Let $\Delta$ be a pure simplicial complex such that ${\rm ht} I_{\Delta}=2$. Then $\mbox{\rm v}( I_{\Delta})=\min \{j \,\,: \beta_{2,2+j}(S/ I_{\Delta})\ne 0\}$.
\end{thm}
\noindent It means that  the v-number of the cover ideal of a graph is just two less than the initial degree of the its  syzygy module, which is a refined version of a result of  Saha \cite[Theorem 3.8]{S} that the v-number of the cover ideal of a graph is bounded above by the regularity of its quotient ring. Moreover, we give formulas for {\rm v}-numbers of cover ideals of multi-whisker graphs (Corollary \ref{multi-whisker}) and very well-covered graphs (Theorem \ref{very well-covered}). In the rest of this section, we discuss the result by Saha and Sengupta \cite[Proposition 4.1]{SS}. This result had a counter-example, so we revise their claim using the notion of nerve complexes (Proposition \ref{line graphs}).

In Section 6, we discuss a relationship between the v-number of the Stanley--Reisner ideal of a simplicial complex and the Serre-depth, which is introduced in \cite{PPTY1, PPTY2}, for the quotient ring of the second symbolic power of the Stanley--Reisner ideal of its Alexander dual: 
\begin{thm}[\mbox{\rm see Theorem \ref{S$_h$-depth}}]
Let $\Delta$ be a pure simplicial complex such that ${\rm ht} I_{\Delta}=h\geq 2$.
Then $$\mbox{\rm v}_P( I_{\Delta})\le n- \mbox{\rm S$_h$-depth}\hspace{0.05cm}S/I_{\Delta^*}^{(2)}-1
\mbox{ \quad for any }P\in \Ass S/I_{\Delta}.$$
\end{thm}
In Section 7, we show that the v-number of  the Stanley--Reisner ideal of a 2-pure simplicial complex is equal to the dimension of its  Stanley--Reisner ring (Theorem \ref{w2}), which is a generalization of  \cite[Theorem 4.5]{JV} on the edge ideal case.
As its corollary, we determine  the v-number of  the Stanley--Reisner ideal of a 2-Cohen--Macaulay, Gorenstein or a matroid  complex (Corollary \ref{2-CM}).

In Section 8, we consider the range of the value of the v-number of Stanley--Reisner ideals and answer some questions posed by \cite{S, SS}.

In the final section of this paper, we show that the v-number of the edge ideal of a very well-covered graph or a multi-whisker graph is bounded above by the regularity of its quotient ring. For very well-covered graphs, this result is well known, as stated in \cite[Theorem 4.3]{GRV}. Here, we provide a new proof of this result using the technique from \cite[Theorem 3.5]{KPTY}.
%%%%%%%%%%%%%%%%%%%%%%%%%%%%%%%%%%%%%%%%%%%%%%%%%%%%%%%%%%%%%%%%%%%%%%

\section{Preliminaries}
In this section, we give several definitions and basic properties of simplicial complexes and simple graphs. Readers familiar with this material may wish to continue to the next section. 

\subsection{Simplicial complexes}
We recall several notations on simplicial complexes and Stanley--Reisner ideals. We refer the reader to \cite{BH, St} for the detailed information on combinatorial and algebraic background. For a positive integer $n$, we set $V=[n]=\{1, 2, \ldots, n\}$. A nonempty subset $\Delta$ of the power set $2^V$ of $V$ is called a \textit{simplicial complex} on $V$, if $\{v\} \in \Delta$ for all $v \in V$, and $F \in \Delta$, $F^{\prime} \subseteq F$ imply $F^{\prime} \in \Delta$. An element $F \in \Delta$ is called a \textit{face} of $\Delta$. The \textit{dimension} of $\Delta$ is defined by $\dim \Delta = \max\{|F|-1\,\,: \text{$F$ is a face of $\Delta$}\}$. A maximal face of $\Delta$ is called a \textit{facet} of $\Delta$. If any facets of $\Delta$ have the same cardinality, then $\Delta$ is called \textit{pure}. We denote the set of facets by $\mathcal{F}(\Delta)$. A vertex $v$ of $\Delta$ is called a \textit{free vertex}, if $v$ is contained only one facet. If we write the vertex set as $V = \{ x_1, \ldots, x_n \}$, then for each $1 \leq i \leq h$, we define $V|_{[h] \setminus \{i\}} = \{ x_1, \ldots, \hat{x}_i, \ldots, x_h \},$ where the hat $\hat{x}_i$ indicates that the element $x_i$ is removed. For a face $F \in \Delta$, the \textit{link} and  \textit{star} of $F$ are defined by
\begin{eqnarray*}
\link_{\Delta} F  &=& \{F^{\prime} \in \Delta\;:\,F^{\prime} \cup F \in \Delta,\, F^{\prime} \cap F = \emptyset\},\\
\star_{\Delta} F  &=& \{F^{\prime} \in \Delta\;:\,F^{\prime} \cup F \in \Delta \}.
\end{eqnarray*}
Suppose $\Delta$ is connected with $\dim \Delta=1$ and let $p$, $q$ be two vertices. The \textit{distance} between $p$ and $q$, denoted by ${\rm dist}_{\Delta}(p, q)$, is the minimal length of paths from $p$ to $q$. The \textit{diameter}, denoted by $\diam \Delta$, is the maximal distance between two vertices in  $[n]$. We set $\diam  \Delta= \infty$, if $\Delta$ is  disconnected. For a $(d-1)$-dimensional case with $d\geq 3$, we define  $\diam \Delta = \diam \Delta ^{(1)}$, where $\Delta ^{(1)}$ is the 1-skeleton $\{F \in \Delta\;:\ \dim F \le 1\}$ of  $\Delta$, which is 1-dimensional. Now, we recall the notion of the Alexander dual complex. The \textit{Alexander dual complex} of $\Delta$ is defined  as follows: $$ \Delta^{*} = \{F \in 2^{V}\,:\, V \setminus F \notin \Delta\}.$$
Then $\Delta^{*}$ is a simplicial complex on the vertex set $V$ and $(\Delta^{*})^{*} = \Delta$. We have the useful tool about the simplicial homologies:  
\begin{thm}\cite[Proposition 1]{ER} \label{LAD}$($``local Alexander duality theorem"$)$ For every $W \subset V$ and every $i$, we have 
$$\widetilde{H}_{i-2}(\link _{\Delta}(V \setminus W);k) \cong \widetilde{H}_{|W|-i-1}((\Delta^{*})_W;k).$$
\end{thm}
Also, $\Delta$ is called a \textit{matroid complex}, if the induced subcomplex $\Delta_{W} = \{ F \in \Delta : F \subset W \}$ is pure for any $W \subset [n]$. If a $(d-1)$-dimensional matroid complex $\Delta$ satisfies $\widetilde{H}_{d-1}(\Delta ; k) \neq 0$, then $\Delta$ is called \textit{matroid*}. The \textit{Stanley--Reisner ideal} of $\Delta$, denoted by $I_{\Delta}$, is the squarefree monomial ideal generated by $$\{x_{i_1} x_{i_2} \cdots x_{i_p} \,:\, 1 \le i_1 < \cdots < i_p \le n,\; \{x_{i_1},\ldots,x_{i_p}\} \notin \Delta \},$$ and $k[\Delta]= k[x_1,\ldots,x_n]/I_{\Delta}$ is called the \textit{Stanley--Reisner ring} of $\Delta$. The \textit{initial degree}, denoted by $\indeg I_{\Delta}$, is the smallest degree of generators which belong to $\mathcal{G}(I_{\Delta})$, where $\mathcal{G}(I_{\Delta})$ is the set of minimal generators of $I_{\Delta}$. Notice that 
$$\indeg I_{\Delta}= \min \{ |F|\;:\ F \text{ is a (minimal) non-face of }\Delta \}.$$
The \textit{big hight} of $I_{\Delta}$, denoted by ${\rm bight}I_{\Delta}$, is the number ${\rm max} \{ {\rm ht}P : P \in {\rm Ass}I_{\Delta} \}$. Since a facet $F$ of $\Delta $ corresponds to the minimal non-face $V \setminus F$ of $\Delta^{*}$, we have :
\begin{prop}$($c.f. e.g.\cite[Proposition1.1]{T}$)$ \label{A-dualideal}
Let $$I_{\Delta}= \bigcap_{ F \in \mathcal{F} (\Delta )}P_{F}$$ 
be the minimal prime decomposition of $I_{\Delta}$, where $P_{F}$ is the monomial prime ideal $( x_j \,\,: j \in V \setminus F)$. Then we have $$I_{\Delta^{*}}=( x^{(V \setminus F)}\;:\ F \in \mathcal{F} (\Delta )),$$ where $x^{(V \setminus F)}$ is the monomial $\prod_{j \in V \setminus F} x_j$. In particular,  $\indeg I_{\Delta^{*}}={\rm ht} I_{\Delta}$. 
\end{prop}
We say that  a simplicial complex $\Delta$ is \textit{Cohen--Macaulay} (resp. \textit{Gorenstein}, \textit{level}), if $k[\Delta]$ is Cohen--Macaulay (resp. Gorenstein, level). Moreover, $\Delta$ is called \textit{2-Cohen--Macaulay} (resp. \textit{2-pure}), if  $\Delta$ is Cohen--Macaulay (resp. pure) and  $\Delta_{V\setminus \{x\}}$ is Cohen--Macaulay (resp. pure) with $\dim \Delta = \dim \Delta _{V\setminus \{x\}}$ for any $x\in V$. A Gorenstein complex $\Delta$ is said to be \textit { Gorenstein* }, if  $x_i$ is a zero-divisor on  $k[\Delta]$ for every $i$. 

\subsection{Graphs and Clutters}
Let $G$ be a graph, which means a finite simple graph, namely, a finite undirected graph without loops or multiple edges in this paper. Let $V(G)$ (resp. $E(G)$) denote the set of vertices (resp. edges) of $G$ and set $V(G)=\{x_{1},\ldots, x_{n}\}$. Then the \textit{edge ideal} of $G$, denoted by $I(G)$, is a squarefree monomial ideal of $S=K[x_1,\ldots,x_n]$ defined by $$I(G) = (x_ix_j \,:\, \{x_i,x_j\} \in E(G)).$$ A subset $A$ of $V(G)$ is an \textit{independent set} of $G$, if no two vertices contained $A$ are adjacent. An independent set is called \textit{maximal}, if it is maximal with respect to inclusion.  Also, we let $i(G)=\min\{|A|\,\,: A\mbox{ is a maximal independent set}\}$. The \textit{independence complex}  $\Delta(G)$ of $G$ is defined as the simplicial complex whose faces are the independent set of $G$, that is, $$\Delta(G)=\{F \subset V  \,:\, e \not\subseteq F \ \mbox{\rm for any } e \in E(G) \},$$ Notice that $I(G)=I_{\Delta(G)}$(see, for example, \cite[Lemma 31]{T13}). A subset $W$ of $V(G)$ is a \textit{vertex cover} of $G$, if for any edge $\{u,v\}\in E(G)$, $u\in W$ or $v\in W$. A vertex cover is called \textit{minimal}, if it is minimal with respect to inclusion. The \textit{cover ideal}, denoted by $J(G)$, of $G$ is the monomial ideal $$J(G)=(x^{W}\,\,: W \mbox{ is a minimal vertex cover of } G).$$ It is known that $J(G)=I_{\Delta(G)^{*}}=I(G)^{*}$ (see, for example, \cite[Corollary 38]{T13}). Next, we explain the notion of clutters which is defined as a generalization of finite simple graphs. A \textit{clutter} $\mathcal{C}$ with vertex set $V(\mathcal{C})$ is a family of subsets of $V(\mathcal{C})$, called edges, which satisfies that any element does not contain any other elements. For an independent set $A$ of $\mathcal{C}$, the \textit{neighbor set} $N_{\mathcal{C}}(A)$ is the set of vertices $v$ which satisfies that there exist an edge $e$ such that $e \subset A \cup \{ v \}$. For $\mathcal{C}$, $\mathcal{A}_{\mathcal{C}}$ denotes the set of independent sets which neighbor set is a minimal vertex cover of $\mathcal{C}$.

\subsection{Very well-covered graphs and multi-whisker graphs} First, we explain a very well-covered graph. See \cite{CRT, KPTY} in detail. A graph $G$ with no isolated vertices is called \textit{very well-covered} 
if $G$ is well-covered and $|V(G)|=2{\rm ht}I(G)$. Let $H$ be a Cohen--Macaulay very well-covered graph with $2d_0$ vertices. Based on \cite{CRT}, we may assume the following condition:
\vspace{0.1cm}
\begin{itemize}
\item[$(\ast)$] $V(H)=X_{[d_0]}\cup Y_{[d_0]}$, where $X_{[d_0]}=\{x_1,\ldots,x_{d_0}\}$ is a minimal vertex cover of $H$ and $Y_{[d_0]}=\{y_1,\ldots,y_{d_0}\}$ is a maximal independent set of $H$ such that $\{x_1y_1,\ldots,x_{d_0}y_{d_0}\} \subseteq E(H)$.
\end{itemize}
For very well-covered graphs, there is a useful theorem to analyze the structure, which has been proved in \cite[Theorem 3.5]{KPTY}. Before stating this theorem, we will explain the necessary notation. Let $G$ be a graph with $xy\in E(G)$, then, the graph $G^{\prime}$, obtained by replacing the edge $xy$ in $G$ with a complete bipartite graph $K_{i, i}$ is defined as $$V(G')=(V(G)\setminus \{x,y\} )\cup \{x_1,\ldots,x_i\}\cup\{y_1 ,\ldots,y_i\}$$ and  
\[
\begin{array}{lll}
E(G')	& = & E(G_{V(G)\setminus \{x,y\}}) \cup\{x_jy_k\,\,: 1\le j,k\le i\} \\[0.15cm]
& & \hspace{2.65cm}\cup\ \{x_jz\,\,: 1\le j\le i,\ z \in V(G) \setminus \{x,y \},\ xz\in E(G)\} \\[0.15cm]
& & \hspace{2.65cm}\cup\ \{y_jz\,\,: 1\le j\le i,\ z \in V(G) \setminus \{x,y \},\ yz\in E(G)\}.
\end{array}
\]
Let $H$ be a Cohen--Macaulay very well-covered graph with $2d_0$ vertices and assume the condition $(\ast)$. Let $H^{\prime}$ be a graph with the vertex set $$V(H^{\prime})=\bigcup_{i=1}^{d_0}\Big(\big\{ x_{i1},\ldots,x_{in_i}\big\}\cup\big\{y_{i1},\ldots,y_{in_i}\big\}\Big)$$ which is obtained by replacing the edges $x_1y_1,\ldots,x_{d_0}y_{d_0}$ in $H$ with the complete bipartite graphs $K_{n_1,n_1},\ldots,K_{n_{d_0},n_{d_0}}$, respectively. We write $H(n_1,\ldots,n_{d_0})$ for $H^{\prime}$. Under this notation, we have the following theorem. 
\begin{thm}\cite[Theorem 3.5]{KPTY}\label{structure}
Let $G$ be a very well-covered graph on the vertex set $X_{[d]}\cup Y_{[d]}$. Then there exist positive integers $n_1,\ldots,n_{d_0}$ with $\sum_{i\in [d_0] }n_i=d$ and a Cohen--Macaulay very well-covered graph $H$ on the vertex set $X_ {[d_0]}\cup Y_{[d_0]}$ such that $G\cong H(n_1,\ldots,n_{d_0})$.
\end{thm}
Next, we explain the concept of a multi-whisker graph. A multi-whisker graph was introduced by Muta, Pournaki and Terai in \cite{MPT} as a generalization of whisker graphs. Given a graph $G_0$ on the vertex set  $X_{[h]}=\{x_{1}, \ldots, x_{h}\}$ the \textit{whisker graph associated with} $G_{0}$ is the graph $G_{1}=G_{0}[1,\ldots, 1]$ on the vertex set $X_{[h]}\cup Y_{[h]}$, where $Y_{[h]}=\{y_{1},\ldots, y_{h}\}$ and the edge set $$E(G_{1})=E(G_{0})\cup\{\{x_{1},y_{1}\},\ldots, \{x_{h},y_{h}\}\}.$$ Given positive integers $n_{1},\ldots, n_{h}$, the \textit{multi-whisker graph associated with } $G_{0}$ is the graph $G=G_{0}[n_{1},\ldots, n_{h}]$ on the vertex set 
$$X_{[h]}\cup Y, \mbox{ where } Y=\{y_{1,1},\ldots, y_{1, n_{1}}\}\cup\cdots\cup\{y_{h,1},\ldots, y_{h,n_{h}}\}$$ and the edge set $$E(G)=E(G_{0})\cup\{\{x_1,y_{1,1}\},\ldots, \{x_{1},y_{1,n_{1}}\},\ldots, \{x_{h},y_{h,1}\},\ldots, \{x_{h},y_{h,n_{h}}\}\}.$$

\subsection{Serre-depth}
For a finitely generated $S$-module $M$ and for an integer $i$, we set $K_{M}^{i}={\rm Hom}_{k}(H_{\mathfrak{m}}^{i}(M), k)$. Notice that $K_{M}^{i}\simeq{\rm Ext}_{S}^{n-j}(M, S(-n))$. For an unmixed homogeneous ideal $I$ and  $r \ge 2$, the \textit{Serre-depth} for $(S_{r})$ of $S/I$, denoted by $\mbox{\rm $S_r$-depth}\hspace{0.05cm}S/I$, is defined by $$\mbox{\rm $S_r$-depth}\hspace{0.05cm}S/I=\min\big\{j \,\,: \dim K^j_{S/I} \ge j-r+1\big\},$$where ${\rm dim}K_{S/I}^{j}$ is the Krull dimension of $K_{S/I}^{j}$ as an $S$-module.  Also, if $K_{S/I}^j=0$, then we define  ${\rm dim}K_{S/I}^{j}=-\infty$. This notion was introduced in \cite{PPTY1, PPTY2} for Stanley--Reisner ideals. It is known that $S/I$ satisfies $(S_{r})$ is equivalent to $S_{r}\mbox{-}{\rm depth}\hspace{0.05cm}S/I=\dim S/I$ by \cite[Lemma 3.2.1]{s0}. 

%%%%%%%%%%%%%%%%%%%%%%%%%%%%%%%%%%%%%%%%%%%%%%%%%%%%%%%%%%%%%%%%%%%%%%

 \section{Takayama's formula for the simbolic powers of\\ Stanley--Reisner ideals} 
In this section, we discuss about the symbolic powers of Stanley--Reisner ideals. First, we recall the notion of symbolic powers. Given an integer $i\geq 1$ and a homogeneous ideal $I$, the \textit{$i$-th symbolic power of $I$} is defined to be an ideal $$I^{(i)}=S\cap\bigcap_{P\in{\rm Ass}(I)}I^{i}S_{P}.$$If $I$ is a squarefree monomial ideal, then one has $$I^{(i)}=\bigcap_{P\in{\rm Ass}(I)}P^i.$$ Takayama's formula \cite[Theorem 1]{Ta} is very useful to analyze the local cohomology module of the residue ring by monomial ideals. Here we quote a version in \cite{CHHKTT}. For $\boldmath{a}=(a_1, \dots, a_n) \in \ZZ^n$, we set $x^{ \boldmath{a}}=x_{1}^{a_1}\cdots x_{n}^{a_n}$ and $\supp_{+}\boldmath{a}=\{i \,:\, \ a_i>0\}$,  $\supp_{-}\boldmath{a}=\{i \,:\, \ a_i<0\}$. For $F\subset [n]$, we define $S_F=S[x_j^{-1} \,:\, j\in F]$. The next theorem is called Takayama's formula. 
\begin{thm} \cite[Theorem 1]{Ta}
Let $I$ be a monomial ideal in a polynomial ring $S$ and  $\boldmath{a} \in \ZZ^{n}$. Then
\begin{eqnarray*}
\dim_k  H^{i}_{\frak{m}}(S/I)
=\dim_k {\widetilde{H}_{i-|\supp_{-}\boldmath{a}|-1}}\big(\Delta_{\boldmath{a}}(I); k\big),
\end{eqnarray*}
where $ \Delta_{\boldmath{a}}(I)=\{F\setminus \supp_{-}{\boldmath{a}}; \   \supp_{-}{\boldmath{a}} \subset F, x^{\boldmath{a}}    \not\in IS_F \}$.
\end{thm}
By using this, we give a version of Takayama's formula for the residue rings by the symbolic powers of Stanley--Reisner ideals.
\begin{thm} \label{Takayama}
Let $\Delta$ be a  simplicial complex and $\ell$ be a positive integer.  Then the Hilbert series of the local cohomology modules of $S/I_{\Delta}^{(\ell)}$ is given by 
$$F\big(H^{i}_{\frak{m}}(S/I_{\Delta}^{(\ell)}),t\big)=\sum_{F \in \Delta} (\sum_{\boldmath{a}\in \NN^{V(\link_{\Delta}F)}} \dim_k \widetilde{H}_{i-|F|-1}\big(\Delta_{\boldmath{a}}(I_{\link_ {\Delta}F}^{(\ell)});k\big)t^{|\boldmath{a}|}\left(\frac{t^{-1}}{1-t^{-1}}\right)^{|F|}),$$
 where $ \NN^{V(\link_{\Delta}F)}\hookrightarrow  \NN^{V(\Delta)}$ is induced by $ V(\link_{\Delta}F)\hookrightarrow  V(\Delta)$. 
\end{thm}
\begin{proof}
First we show that $\Delta_{\boldmath{a}}(I_{\Delta}^{(\ell)})=\Delta_{\boldmath{a_+}}(I_{\link_ {\Delta}F}^{(\ell)})$ for $\boldmath{a}\in \ZZ^n$. We may assume that $\supp_{-}\boldmath{a} \in \Delta$. Then, we have 
\begin{eqnarray*}
\Delta_{\boldmath{a}}(I_{\Delta}^{(\ell)}) &=& \{G\setminus F \,:\,F\subset G, x^{\boldmath{a}}\not\in I_{\Delta}^{(\ell)}S_G\}\\
&=& \{H\in \link_ {\Delta}F \,:\,  x^{\boldmath{a _{+}}}\not\in I_{\link_ {\Delta}F}^{(\ell)}S_H\}\\
&=& \Delta_{\boldmath{a_+}}(I_{\link_ {\Delta}F}^{(\ell)}),
\end{eqnarray*}
where $F=\supp_{-}\boldmath{a}$. Notice that for $\boldmath{a}, \boldmath{b} \in \ZZ^n$ with $\supp_{+}\boldmath{a}=\supp_{+}\boldmath{b}$ and $\supp_{-}\boldmath{a}=\supp_{-}\boldmath{b}$, we have $\Delta_{\boldmath{a}}(I_{\Delta}^{(\ell)})=\Delta_{\boldmath{b}}(I_{\Delta}^{(\ell)})$  by the definition. Each variable corresponding to a vertex in $F$ can move from $-1$ to $-\infty$. Hence  for $F \in \Delta$, 
\begin{eqnarray*}
&&\sum_{\supp_{-}\boldmath{a}=F } \dim _{k}\big[H^{i}_{\frak{m}}(S/I_{\Delta}^{(\ell)})\big]_{\boldmath{a}}\ t^{|\boldmath{a}|}\\[-0.2cm]
&=&\sum_{\boldmath{a}\in \NN^{V(\link_{\Delta}F)}} \dim_k \widetilde{H}_{i-|F|-1}\big(\Delta_{\boldmath{a}}(I_{\link_ {\Delta}F}^{(\ell)});k\big)t^{|\boldmath{a}|}\left(\frac{t^{-1}}{1-t^{-1}}\right)^{|F|}.
\end{eqnarray*}
Therefore, we obtain that  

$$F\big(H^{i}_{\frak{m}}(S/I_{\Delta}^{(\ell)}),t\big)
= \sum_{F \in \Delta} (\sum_{\boldmath{a}\in \NN^{V(\link_{\Delta}F)}} \dim_k \widetilde{H}_{i-|F|-1}\big(\Delta_{\boldmath{a}}(I_{\link_ {\Delta}F}^{(\ell)});k\big)t^{|\boldmath{a}|}\left(\frac{t^{-1}}{1-t^{-1}}\right)^{|F|}).$$
\end{proof}
As a corollary of Theorem \ref{Takayama}, we obtain the formula for the Hilbert series of $K^{i}_{S/I_{\Delta}^{(\ell)}}$. 
\begin{cor} \label{dual}
Let $\Delta$ be a  simplicial complex.  Then
\begin{eqnarray*}
F\big(K^{i}_{S/I_{\Delta}^{(\ell)}},t\big)
& = & \sum_j \dim _{k}\big[K^{i}_{S/I_{\Delta}^{(\ell)}}\big]_{j}\ t^{j}\\[-0.2cm]
& = & \sum_{F \in \Delta} \big(\sum_{\boldmath{a}\in \NN^{V(\link_{\Delta}F)}} \dim_k \widetilde{H}_{i-|F|-1}\big(\Delta_{\boldmath{a}}(I_{\link_ {\Delta}F}^{(\ell)});k\big)t^{-|\boldmath{a}|}\left(\frac{t}{1-t}\right)^{|F|}\big),
\end{eqnarray*}
where $ \NN^{V(\link_{\Delta}F)}\hookrightarrow  \NN^{V(\Delta)}$ is induced by $ V(\link_{\Delta}F)\hookrightarrow  V(\Delta)$. 
\end{cor}

%%%%%%%%%%%%%%%%%%%%%%%%%%%%%%%%%%%%%%%%%%%%%%%%%%%%%%%%%%%%%%%%%%%%%%

\section{the characterization of the {\rm v}-numbers via the Alexander dual}
In this section, we characterize the v-number of the  Stanley--Reisner ideal of a simplicial complex in terms of its Alexander dual. For this purpose, we prepare the next lemma.   
\begin{lemma}\label{cha_lem}
Let $\Delta$ be a simplicial complex on the vertex set $V$ and let $P$ be an associated prime of  $I_{\Delta}$. Then $$\mbox{\rm v}_P( I_{\Delta})=\min \{ |C| \,:\, I_{\Delta^{*}} \not\subset (x_i : x_i \in C) , I_{\Delta^{*}}\subset  (x_j)+ (x_i : x_i \in C)  \mbox { for all } x_j\in P\}.$$
\end{lemma}
\begin{proof} 
Notice that for $C\subset  V=\{x_1, \dots, x_n\}$, $x_C\in I_{\Delta}$ if and only if $x_C$ is divisible by some  $x_D\in\mathcal{G}(I_{\Delta}).$ The latter one is equivalent to $(x_i : x_i \in C)$ contains some associated prime $(x_i : x_i \in D)$ of $I_{\Delta^{*}}.$ Moreover it is equivalent to $I_{\Delta^{*}}\subset  (x_i : x_i \in C).$ Hence, by \cite[Lemma 3.4(c)]{JV}, we obtain that 
\begin{align*}
&\mbox{\rm v}_P( I_{\Delta}) \\
&=\min \{ |C| \,:\,  I_{\Delta}:x_C=P\} \\
&=\min \{ |C| \,:\, x_C\not\in I_{\Delta}, x_jx_C \in I_{\Delta}  \mbox { for all } x_j\in P\} \\
&=\min \{ |C| \,:\, I_{\Delta^{*}} \not\subset (x_i : x_i \in C) , I_{\Delta^{*}}\subset(x_j)+ (x_i : x_i \in C)\mbox { for all } x_j\in P\},
\end{align*}
which completes the proof. 
\end{proof}
Now  we characterize the v-number of the  Stanley--Reisner ideal of a simplicial complex in terms of its Alexander dual.
\begin{thm}\label{intersection}
Let $\Delta$ be a simplicial complex on the vertex set $ V=\{x_1, \dots, x_n\}$. Assume $P=(x_1, x_2, \dots,  x_h)$ be an associated prime of  $I_{\Delta}$. Then $$\mbox{\rm v}_P( I_{\Delta})=n-h -\max \{ |\cap_{j=1}^{h} F_j | \,:\, F_j \in \mathcal{F}(\Delta^{*}) \mbox{ such that }V|_{[h]\setminus{\{i\}}}\subset F_j \mbox{ for all } j \},$$
where $V|_{[h]\setminus{\{i\}}}=\{x_{1},\ldots, \hat{x_{i}}, \ldots, x_{h}\}$
\end{thm}
\begin{proof} 
Notice that  $\{x_1, \dots,  x_h\} \not\in \Delta^{*}$. By Lemma \ref{cha_lem},  we have $$\mbox{\rm v}_P( I_{\Delta})=\min \{ |C| \,:\, I_{\Delta^{*}} \not\subset (x_i : x_i \in C) , I_{\Delta^{*}}\subset  (x_j)+ (x_i : x_i \in C)  \mbox { for all } x_j\in P\}.$$ Now, since  $I_{\Delta^{*}} \not\subset (x_i : x_i \in C)$ and $I_{\Delta^{*}}\subset  (x_j)+ (x_i : x_i \in C)$ for all $x_j\in P$ if and only if $V\setminus C \not\in \Delta^{*}$ and $V\setminus (C  \cup \{ x_j\})\in \Delta^{*} $ for all $1\le j\le h$, by taking $F\subset V$ such that $V\setminus F=C$, we obtain that $$\mbox{\rm v}_P( I_{\Delta})=n-\max \{ |V\setminus F| \,:\, F \not\in \Delta^{*} ,  F \setminus \{ x_j\}\in \Delta^{*}  \mbox { for all } 1\le j\le h \}.$$ Therefore, by letting $F_{j}=F\setminus\{x_{j}\}$, we have $|F\setminus\{x_{1},\ldots, x_{h}\}|=|\cap_{j=1}^{h} F_{j} |$, the assertion follows.  
\end{proof}

Thanks to Theorem \ref{intersection}, we obtain the following formula for cover ideals: 
\begin{cor}\label{cover ideals}
For a graph $G$ with $E(G)\neq\emptyset$, we have 
$${\rm v}(J(G))=|V(G)|-\max\{|F\cap F^{\prime}|\,\,: F, F^{\prime}\in\mathcal{F}(\Delta(G))\mbox{ with }F\neq F^{\prime}\}-2.$$
\end{cor}
In the rest of this section, we give simple corollaries derived from Theorem \ref{intersection}. 
\begin{cor}
Let $\Delta$ be a simplicial complex on the vertex set $V$. Then $$\mbox{\rm v}_h( I_{\Delta}) \ge \min \{\deg \lcm(m_1,m_2, \ldots, m_h) \,:\,m_1, m_2, \ldots, m_h \in \mathcal{G}(I_{\Delta}) \mbox{ are distinct}\}-h$$for  an integer $h\ge 1$.
\end{cor}

\begin{proof} By the definition of ${\rm v}_{h}(I_{\Delta})$ and Theorem \ref{intersection}, we have
\begin{eqnarray*}
\mbox{\rm v}_h( I_{\Delta}) &=&\min \{\mbox{\rm v}_P( I_{\Delta}) \,:\, {\rm ht} P=h\}\\
& \ge & n-h -\max \{ |\cap_{i=1}^{h} F_i | \,:\, F_1,F_2, \ldots, F_h  \in \mathcal {F}(\Delta^{*})\mbox{ are distinct} \} \\
& = & \min \{ |\cup_{i=1}^{h} (V\setminus F_i) | \,:\, F_1,F_2, \ldots, F_h  \in \mathcal {F}(\Delta^{*})\mbox{ are distinct}\} -h,
\end{eqnarray*}
which completes the proof. 
\end{proof}

\begin{cor}
Let $\Delta$ be a  pure simplicial complex with ${\rm ht} I_{\Delta}=h\ge 1$. If $\mbox{\rm v}( I_{\Delta})=\indeg I_{\Delta}-1$, then $\beta_{i, \indeg  I_{\Delta}+i-1}(S/I_{\Delta})\ne 0$ for $1\le i \le h$.
\end{cor}

%%%%%%%%%%%%%%%%%%%%%%%%%%%%%%%%%%%%%%%%%%%%%%%%%%%%%%%%%%%%%%%%%%%%%%

\section{the {\rm v}-numbers of Stanley--Reisner ideals with pure height two}

In this section, we study the {\rm v}-number of the Stanley--Reisner ideal with height$\leq 2$. First, we describe the {\rm v}-number of the Stanley--Reisner ideal with height 1. 
\begin{prop}\label{height1}
Let $\Delta$ be a simplicial complex such that ${\rm ht} I_{\Delta}=1$. Then
$$\mbox{\rm v}( I_{\Delta})=\min \{j \,:\, \beta_{1,j}(S/ I_{\Delta})\ne 0\}-1.$$
\end{prop}
\begin{proof}
We may assume that $(x_1)$ is an associated prime of $I_{\Delta}$.
Then we can take $x_1m\in \mathcal{G}(I_{\Delta})$ such that $\deg m \le \deg m'$ for any  $x_1m'\in \mathcal{G}(I_{\Delta})$. Then $I_{\Delta}: m= (x_1)$. Hence  ${\rm v}( I_{\Delta}) \le \deg m$. Also, since ${\rm v}( I_{\Delta})\ge \indeg I_{\Delta}-1=\deg m$ by \cite[Lemma 3.17]{JV}, we have ${\rm v}( I_{\Delta})=\deg m$.
\end{proof}
We give the relation the {\rm v}-number with the graded Betti number of the Stanley--Reisner ideal with height 2. 
\begin{thm}\label{2-nd betti}
Let $\Delta$ be a simplicial complex such that ${\rm ht} I_{\Delta}=2$. If there is an associated prime ideal $P \in \Ass S/I_{\Delta}$ with ${\rm ht}P=2$ such that ${\rm v}( I_{\Delta})={\rm v}_P(I_{\Delta})$, then $\beta_{2,2+{\rm v}( I_{\Delta})}(S/ I_{\Delta})\ne 0$.
\end{thm}
\begin{proof}  
Since ${\rm ht}P=2$, we can write $P=(x_i, x_j)$. Also, by assumption, there are facets $F, F' \in \Delta^{*}$  such that  $x_i\in F$, $x_j\in F'$ and ${\rm v}( I_{\Delta})=\mbox{\rm v}_{(x_i, x_j)}( I_{\Delta})=n- | F \cap  F'|-2$. We show $\tilde{H}_0(\link_{\Delta^{*}}(G); k)\ne 0$, where we set $G=F \cap  F'$. To this end, it suffices to show that $x_i$ does not have a path to  $x_j$ in ${\rm link}_{\Delta^{*}}(G)$. We proceed by contradiction. Then, there exists a shortest path  $x_i=x_{p_0}, x_{p_1}, \dots, x_{p_m}=x_j$ from $x_i$ to  $x_j$, where $m\ge 2$. Since the path is the shortest, $\{x_{p_0}, x_{p_2}\} \not\in \link_{\Delta^{*}}(G)$. This implies that there exists a face $H\subset G$ such that $x_{p_0}x_{p_2}\prod_{x_{\ell}\in H}x_{\ell}$ is a minimal generator of   $I_{\Delta^{*}}$. Hence we have 
$$\{(G\cup\{x_{p_{1}},x_{p_{2}}\}), (G\cup\{x_{p_{1}},x_{p_{0}}\})\}\cup\{(G\setminus\{x_{\ell}\})\cup\{x_{p_{0}},x_{p_{1}},x_{p_{2}}\}\,\,: x_{\ell}\in H\}\subset\Delta^{*}.$$
Notice that these intersections are $G\cup \{x_{p_1}\} \setminus H$. 
This means that, by Theorem \ref{intersection}, 
\begin{eqnarray*}
\mbox{\rm v}_{|H|+2}( I_{\Delta}) &\le&  n-|G\cup \{x_{p_1}\} \setminus H|-(|H|+2)\\
&=&n-|G|-3\\
&=&\mbox{\rm v}( I_{\Delta})-1,
\end{eqnarray*}
which  is a contradiction. Thus, by Theorem \ref{LAD} , we have $$0\ne \tilde{H}_0(\link_{\Delta^{*}}(G); k) \cong \tilde{H}_{|V\setminus G|-3}(\Delta_{V\setminus G}; k).$$ Then  by  Hochster's formula on graded betti numbers, we obtain that $$0\ne \dim_k\tilde{H}_{|V\setminus G|-3}(\Delta_{V\setminus G}; k)=\beta_{2,|V\setminus G|}(S/I_{\Delta}).$$ Therefore $\beta_{2,2+{\rm v}( I_{\Delta})}(S/I_{\Delta})\ne 0$.
\end{proof}
Thanks to Theorem \ref{2-nd betti}, we have the following corollary, which describes the {\rm v}-number of the Stanley--Reisner ideal with pure height two. 
\begin{cor}\label{pureheight2}
Let $\Delta$ be a pure simplicial complex such that ${\rm ht} I_{\Delta}=2$. Then $$\mbox{\rm v}( I_{\Delta})=\min \{j \,\,: \beta_{2,2+j}(S/ I_{\Delta})\ne 0\}.$$
\end{cor}
\begin{proof}  
Let $P\in{\rm Ass}S/I_{\Delta}$ with ${\rm v}(I_{\Delta})={\rm v}_{P}(I_{\Delta})$. Then $\beta_{2,2+{\rm v}(I_{\Delta})}(S/I_{\Delta})\neq 0$ by Theorem \ref{2-nd betti}. Now, for any $G\in\Delta^{*}$ with $|G|>n-{\rm v}(I_{\Delta})-2$, $\link_{\Delta^{*}}(G)$ is $\{\emptyset\}$ or a simplex , since any two vertices make a 1-face and $\link_{\Delta^{*}}(G)$  is a flag complex. Thus, $ \tilde{H}_0(\link_{\Delta^{*}}(G); k)=0$. Therefore, by Theorem \ref{LAD}, we have $$0=\tilde{H}_0(\link_{\Delta^{*}}(G); k) \cong \tilde{H}_{|V\setminus G|-3}(\Delta_{V\setminus G}; k).$$ Then  by  Hochster's formula on graded betti numbers, we obtain that $$0=\dim_k \tilde{H}_{|V\setminus G|-3}(\Delta_{V\setminus G}; k)=\beta_{2, |V\setminus G|}(S/I_{\Delta}).$$
\end{proof}
The following corollaries follow immediately from Corollary \ref{pureheight2}.
\begin{cor}
Let $\Delta$ be a pure simplicial complex such that ${\rm ht} I_{\Delta}=2$ and set 
$$p=\min \{\deg \lcm(m,m')\,\: m, m'\in \mathcal{G}(I_{\Delta}), m \ne m'\}$$
Then $\mbox{\rm v}( I_{\Delta})=p-2$.
\end{cor}

\begin{cor}\label{level}
Let $\Delta$ be a  Cohen--Macaulay simplicial complex with ${\rm ht} I_{\Delta}=2$. Then we have $\mbox{\rm v}( I_{\Delta})= \reg(S/ I_{\Delta})$  if and only if  $\Delta$ is level.
\end{cor}
\begin{proof}
In this case, $\min \{j \,\,: \beta_{2,2+j}(S/ I_{\Delta})\ne 0\}={\rm reg}\hspace{0.05cm}{S/I_{\Delta}}$ if and only if $\Delta$ is level. Thus, the assertion follows by Corollary \ref{pureheight2}. 
\end{proof}
We reprove the Saha's main result in \cite[Theorem 3.8]{S} from Corollary \ref{pureheight2}. 
\begin{cor}\cite[Theorem 3.8]{S}\label{regularity}
Let $\Delta$ be a pure simplicial complex such that ${\rm ht} I_{\Delta}=2$. Then $\mbox{\rm v}( I_{\Delta}) \le \reg(S/ I_{\Delta})$.
\end{cor}
The next example shows that the conclusion  of the above corollary does not hold in pure height 3 case.
\begin{exam}[{\bf The 6-vertex  triangulation of the real projective plane}] \label{RealP}
Let $\Delta$ be  the 6-vertex triangulation of the real projective plane $\mathbb{P}^2(\RR)$.Then $\reg S/I_{\Delta}=2$ if the characteristic of  the base field is not 2. On the other hand, $\mbox{\rm v}( I_{\Delta})=3$. This follows from Theorem \ref{w2}, since $\Delta$ is a 2-dimensional 2-pure simplicial complex. Hence in the  height 3 case  $\mbox{\rm v}( I_{\Delta}) \le \reg(S/ I_{\Delta})$ does not hold in general.
\end{exam}

\begin{thm}\label{sequentially CM}
Let $\Delta$ be a simplicial complex on the vertex set $[n]$ with ${\rm ht} I_{\Delta}=2$ such that its Alexander dual $\Delta^{*}$ is sequentially Cohen--Macaulay and let $F_1, F_2, \dots , F_m$ be the facets of $\Delta^{*}$ with $d_i=\dim F_i+1$. If $m\ge 2$ and  $d_1\ge d_2 \ge \cdots \ge d_m$, then we have $\mbox{\rm v}(I_{\Delta^{*}})=n-d_2-1$.
\end{thm}
\begin{proof} 
Notice that  $|F_{i}\cap F_{j}|\le d_2-1$ for any facets $F_{i}, F_{j}$ with $i\neq j$. Since the pure skeleton  $(\Delta^{*})^{[d_2-1]}$ is Cohen--Macaulay,  $(\Delta^{*})^{[d_2-1]}$ is connected in codimension one, that is, there exists a facet $E$  of $(\Delta^{*})^{[d_2-1]}$ such that $|E\cap F_2|=d_2-1$. Hence, there exists  a facet $F$  of $\Delta^{*}$ such that $|F\cap F_2|=d_2-1$. Therefore, we have  $\mbox{\rm v}(I_{\Delta^{*}})=n-(d_2-1)-2=n-d_2-1$.
\end{proof}
Thanks to Theorem \ref{sequentially CM}, we obtain the result about the  {\rm v}-number of the cover ideal of a multi-whisker graph. 
\begin{cor}\label{multi-whisker}
Let $G_{0}$ be a graph on the vertex set $X_{[h]}$ and let $G=G_0[n_1, n_2, \dots , n_{h}]$ be the multi-whisker graph associated with $G_{0}$. Then we have $$\mbox{\rm v}( J(G))=h+\min\{n_1, n_2, \dots , n_{h}\}-2.$$
\end{cor}
\begin{proof} 
Notice that $G$ is sequentially Cohen--Macaulay, by \cite[Theorem 6.1]{MPT}. The largest facet of $\Delta(G)$ has dimension $n_1+n_2+\cdots +n_{h}$ and the second largest one has dimension $(n_1+n_2+\cdots +n_{h})-\min\{n_1, n_2, \dots , n_{h}\}+1$. Hence, by Theorem \ref{sequentially CM}, the assertion follows. 
\end{proof}
Moreover, we describe the {\rm v}-numbers of the cover ideal of  a very well-covered graph.
\begin{thm}\label{very well-covered}
Let $G=H(n_1, n_2, \dots , n_{h_0})$ be a very  well-covered graph with ${\rm ht}I(G)=h$, where $H$ is a Cohen--Macaulay very well-covered graph. Then we have $$\mbox{\rm v}( J(G))=h+\min\{n_1, n_2, \dots , n_{h_0}\}-2.$$
\end{thm}
\begin{proof} 
First notice that $\mathcal{F}(\Delta(H))$ and  $\mathcal{F}(\Delta(G))$ have one-to-one correspondence. In fact, $F=\{z_1, z_2, \dots, z_{h_0}\}\in \mathcal{F}(H)$ corresponds to $\tilde{F}=Z_1 \cup Z_2 \cup \cdots \cup Z_{h_0} \in \mathcal{F}(\Delta(G))$, where $z_i\in \{x_i, y_i\}$ and   $X_i=\{x_{i1}, x_{i2}, \dots, x_{in_i}\}$,  $Y_i = \{y_{i1}, y_{i2}, \dots, y_{in_i}\}$ and  $Z_i  \in  \{X_i, Y_i\}$. Hence we see that $ |\tilde{F}\cap \tilde{F'}| \le h-\min\{n_1, n_2, \dots , n_{h_0}\}$ for any $F, F' \in  \mathcal{F}(\Delta(H))$. We suppose that $n_p=\min\{n_1, n_2, \dots , n_{h_0}\}$ and take two facets $F, F' \in \mathcal{F}(\Delta(H))$ such that $x_p\in F\setminus F^{\prime}$ ,  $y_p\in F'\setminus F$. Since $H$ is Cohen--Macaulay, $\Delta(H)$ is connected in codimension one in $H$, that is,  there exists a chain of facets $F=F_0, F_1, \dots , F_{\ell}=F'$ such that $ |F_{i-1}\cap F_{i}|=h_0-1$ for all $i=1,2, \dots , \ell$. Hence there exists $1\le q\le \ell$ such that   $F_{q-1}\setminus  F_{q}=x_p$, and thus we have  $ |\tilde{F}_{q-1}\cap \tilde{F}_{q}|=h-n_p$ in $G$. Therefore, we have  $\mbox{\rm v}( J(G))=2h- (h-n_p)-2=h+n_p-2$.
\end{proof}
In the rest of this section, we provide a counter-example and a correct version of  \cite[Proposition 4.1]{SS}. Before explaining this, we recall definitions of cliques and line graphs. Let $G$ be a graph. Then $C\subset V(G)$ is called a \textit{clique} if $\{ u, v \} \in E(G)$ for any $u, v \in C$ with $u \neq v$. Also, the \textit{line graph} $L(G)$ of $G$ is the graph whose vertex set is $E(G)$ and $\{ e, e^{\prime} \}\in E(L(G))$ if $e, e^{\prime}$ satisfy $e \cap e^{\prime} \neq \emptyset$. Saha and Sengupta stated the following proposition in \cite[Proposition 4.1]{SS}; 
\begin{prop}\label{mistake}\cite[Proposition 4.1]{SS}
Let $G$ be a graph and $L(G)$ be its line graph. Suppose that $c(L(G))$ denotes the minimum number of cliques in $L(G)$, such that any vertex of $L(G)$ is either a vertex of those cliques or adjacent to some vertices of those cliques. Then ${\rm v}(I(G))=c(L(G))$. 
\end{prop}
We provide the counter-example to this statement. 
\begin{exam}
Let $G$ be a graph on the vertex set $[10]$ with the edge $e_{1} = \{ 1, 2 \}, e_{2} = \{ 2, 5 \}, e_{3} = \{ 4, 5 \}, e_{4} = \{ 3, 4 \}, e_{5} = \{ 5, 6 \}, e_{6} = \{ 6, 7 \}, e_{7} = \{ 7, 8 \}, e_{8} = \{ 6, 9 \}$ and $e_{9} = \{ 9, 10 \}$. Then $G$ and its line graph $L(G)$ are 

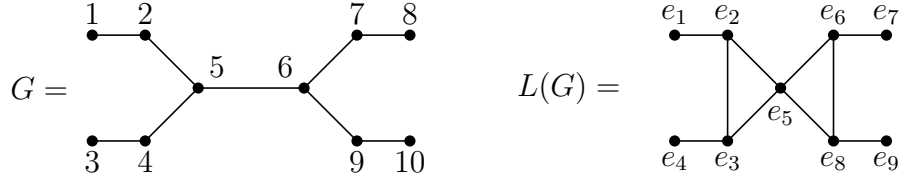
\begin{center}
\begin{tikzpicture}[scale=0.7]
\coordinate (1) at (-3,1);
\fill (1) circle (3pt); 
\path (1) node [above] {1};
\coordinate (2) at (-2,1);
\fill (2) circle (3pt); 
\path (2) node [above] {2};
\coordinate (3) at (-3,-1);
\fill (3) circle (3pt); 
\path (3) node [below] {3};
\coordinate (4) at (-2,-1);
\fill (4) circle (3pt); 
\path (4) node [below] {4};
\coordinate (5) at (-1,0);
\fill (5) circle (3pt); 
\path (5) node [above right] {5};
\coordinate (6) at (1,0);
\fill (6) circle (3pt); 
\path (6) node [above left] {6};
\coordinate (7) at (2,1);
\fill (7) circle (3pt); 
\path (7) node [above] {7};
\coordinate (8) at (3,1);
\fill (8) circle (3pt); 
\path (8) node [above] {8};
\coordinate (9) at (2,-1);
\fill (9) circle (3pt); 
\path (9) node [below] {9};
\coordinate (10) at (3,-1);
\fill (10) circle (3pt); 
\path (10) node [below] {10};
\draw [semithick](-3,1)--(-2,1);
\draw [semithick](-2,1)--(-1,0);
\draw [semithick](-1,0)--(-2,-1);
\draw [semithick](-2,-1)--(-3,-1);
\draw [semithick](-1,0)--(1,0);
\draw [semithick](1,0)--(2,1);
\draw [semithick](2,1)--(3,1);
\draw [semithick](1,0)--(2,-1);
\draw [semithick](2,-1)--(3,-1);
\draw (-4,0)node {$G=$};
\coordinate (1) at (8,1);
\fill (1) circle (3pt); 
\path (1) node [above] {$e_{1}$};
\coordinate (2) at (9,1);
\fill (2) circle (3pt); 
\path (2) node [above] {$e_{2}$};
\coordinate (3) at (9,-1);
\fill (3) circle (3pt); 
\path (3) node [below] {$e_{3}$};
\coordinate (4) at (8,-1);
\fill (4) circle (3pt); 
\path (4) node [below] {$e_{4}$};
\coordinate (5) at (10,0);
\fill (5) circle (3pt); 
\path (5) node [below=1mm] {$e_{5}$};
\coordinate (6) at (11,1);
\fill (6) circle (3pt); 
\path (6) node [above] {$e_{6}$};
\coordinate (7) at (12,1);
\fill (7) circle (3pt); 
\path (7) node [above] {$e_{7}$};
\coordinate (8) at (11,-1);
\fill (8) circle (3pt); 
\path (8) node [below] {$e_{8}$};
\coordinate (9) at (12,-1);
\fill (9) circle (3pt); 
\path (9) node [below] {$e_{9}$};
\draw [semithick](8,1)--(9,1);
\draw [semithick](9,1)--(9,-1);
\draw [semithick](9,-1)--(8,-1);
\draw [semithick](9,1)--(10,0);
\draw [semithick](9,-1)--(10,0);
\draw [semithick](10,0)--(11,1);
\draw [semithick](10,0)--(11,-1);
\draw [semithick](11,1)--(11,-1);
\draw [semithick](11,1)--(12,1);
\draw [semithick](11,-1)--(12,-1);
\draw (6,0)node {$L(G)=$};
\end{tikzpicture}
\end{center}

Then, since $c(L(G)) =2\neq 3={\rm v}(I(G))$, this is the counter-example for \cite[Proposition 4.1]{SS}. 
\end{exam}
The following proposition is a correct version of \cite[Proposition 4.1]{SS}, addressing the errors therein by using the notion of nerve complexes which is similar to the notion of line graphs. To this end, let us recall the definition of nerve complexes  of graphs. The \textit{nerve complex} of $G$, denoted by ${\rm Nerve}(G)$, is the simplicial complex on the vertex set $E(G)$ which face is an element $F$ of the power set $2^{E(G)}$ satisfying $\bigcap_{e \in F} e \neq \emptyset$.
\begin{prop}\label{line graphs}
Let $G$ be a graph with $E(G) = \{ e_1, \ldots, e_m \}$ and let $P$ be the set of $\Gamma$ which is a subset of $\mathcal{F}({\rm Nerve}(G))$ satisfying the following conditions;
\begin{enumerate}
\item for any elements $F$ and $F^{\prime}$ of $\Gamma$ with $F \neq F^{\prime}$, $F \cap F^{\prime} = \emptyset$,
\item any vertex $e$ of ${\rm Nerve}(G)$ satisfies either following (a) or (b) :
\begin{enumerate}
\item there exists an element $F$ of $\Gamma$ such that $e \in F$
\item there exist an element $F$ of $\Gamma$ and $e^{\prime} \in F$ such that $\{ e, e^{\prime} \} \in {\rm Nerve}(G)$.
\end{enumerate}
\end{enumerate}
Let $c({\rm Nerve}(G)) = {\rm min} \{ |\Gamma| \,\,: \Gamma \in P \}$.
Then ${\rm v}(I(G)) = c({\rm Nerve}(G))$.
\end{prop}
\begin{proof}
Let $I = I(G)$. By \cite[Theorem 3.5]{JV}, we can take $A \in \mathcal{A}_{G}$ with $|A| = {\rm v}(I)$. For any $x_i$ in $A$, we set $E_{G}(x_i)=\{ e \in E(G) \,\,: x_i \in e \}$ and $E_{G}(A)=\{ E_{G}(x_i) \,\,: x_i \in A \}$. Then we get ${\rm v}(I) = |A| = |E_{G}(A)| \geq c({\rm Nerve}(G))$. Conversely, we take $\Gamma = \{F_1, \ldots, F_r \}$ an element of $P$ with $r = |\Gamma| = c({\rm Nerve}(G))$ and write $F_i = \{ e_{i_1}, \ldots, e_{i_{m_i}} \}$ for any $i$. Since $F_i$ is a facet of ${\rm Nerve}(G)$, we have $\bigcap_{j = 1}^{m_i} e_{i_j} \neq \emptyset$, and hence we can take $x_i \in \bigcap_{j = 1}^{m_i} e_{i_j}$. We set $B=\{ x_1, \ldots, x_r \}$. If there exist $x_i, x_j \in B$ with $i \neq j$ such that $\{ x_i, x_j \}\in E(G)$, then $\{ x_i, x_j \} \in F_i \cap F_j$, which is a contradiction. Thus $B$ is an independent set of $G$. We show that $N_{G}(B)$ is a vertex cover of $G$. Fix $e\in E(G)$. If $e$ belongs to $F_i$ for some $i$, then $e \cap N_{G}(B) \neq \emptyset$. We assume that $e \in E(G) \setminus \bigcap_{i = 1}^{r} F_i$. Then $e$ must satisfy the condition (b), so we can take $e^{\prime} \in F_i$ for some $i$ with $\{ e, e^{\prime} \} \in {\rm Nerve}(G)$, that is, $e \cap e^{\prime} \neq \emptyset$. Now, since $e^{\prime} \in F_i$, we can write $e^{\prime} = \{ x_i, w \}$ for some vertex $w$, and hence $w \in N_{G}(B)$.
 Since $e \in E(G) \setminus \bigcap_{i = 1}^{r} F_i$ and $e \cap e^{\prime} \neq \emptyset$, we have $w \in e$. Therefore, $N_{G}(B)$ is a vertex cover of $G$, and hence $B \in \mathcal{A}_{G}$ from \cite[Lemma  3.4(b)]{JV}, which leads to the conclusion. 
\end{proof}

%%%%%%%%%%%%%%%%%%%%%%%%%%%%%%%%%%%%%%%%%%%%%%%%%%%%%%%%%%%%%%%%%%%%%%

 \section{An upper bound for the ${\rm v}_P$-numbers of Stanley--Reisner ideals} 
In this section, we discuss an upper bound for the ${\rm v}_P$-numbers. First, we provide an upper bound for the ${\rm v}_{P}$-numbers of the Stanley--Reisner ideals in terms of arithmetical degree. Here, for a squarefree monomial ideal $I$, the \textit{arithmetical degree}, denoted {\rm arith-deg}$I$, is equal to the number of associated prime ideals of $I$. 
\begin{prop}\label{}
Let $\Delta$ be a simplicial complex and   $P\in \Ass S/I_{\Delta}$.
Then  $$\mbox{\rm v}_{P}( I_{\Delta}) \le \mbox{\rm arith-deg}I_{\Delta}-1.$$
\end{prop}
\begin{proof}
Notice that $(I_{\Delta}):{\prod_{ Q\in \Ass S/I_{\Delta}\setminus \{P\}}x_{i_Q}}=P$, where we take $x_{i_Q}\in Q\setminus P$ for $Q\in  \Ass S/I_{\Delta}\setminus \{P\}$. Also, since $I_{\Delta}$ has no embedded prime ideals, we have $$\mbox{\rm v}_{P_F}( I_{\Delta}) \le |\{x_{i_Q} \,\,: Q\in \Ass S/I_{\Delta}\setminus \{P\}\}| \le \mbox{\rm arith-deg}I_{\Delta}-1,$$
which completes the proof. 
\end{proof}
Next, we consider the case that the ${\rm v}_{P}$-number is 1. To this end, we give the characterization of the ${\rm v}_{P}$-numbers in terms of simplicial complexes. 
\begin{prop}\label{v_P 1}
Let $\Delta$ be a simplicial complex on the vertex set $V$, $F\in \mathcal{F}(\Delta)$ and let $m$ be an integer with $2\le m \le |V|-\mbox{\rm bight} I_{\Delta}$. Then, $\mbox{\rm v}_{P_F}( I_{\Delta}) \le m $ if and only if  there exists $W \subset F$ with $|W|=m$ such that $|W\cap G|\le m-1$ for any $G \in \mathcal{F}(\Delta)\setminus\{F\}$.
\end{prop}
\begin{proof}
${\rm v}_{P_F}(I_{\Delta})\le m$ if and only if  there exists $W= \{ x_{i_1},  x_{i_2}, \dots,  x_{i_m}\}\subset  V$ such that  $ x_{i_j} \not\in P_F$ for any  $j$ and there exists $ x_{i_{j_Q}}\in W$ such that $ x_{i_{j_Q}}\in Q$ for any $Q\in \Ass S/I_{\Delta}\setminus \{P_F\}$. Since the latter one is equivalent to the condition that there exists $W\subset  F$ with $|W|=m$ such that  $W\not\subset  G$ for any $ G \in \mathcal{F}(\Delta)\setminus\{F\}$, the assertion follows. 
\end{proof}

\begin{cor}\label{v_P 2}
Let $\Delta$ be a simplicial complex on the vertex set $V$, $F\in \mathcal{F}(\Delta)$ and let $m$ be an integer with $1\le m \le |V|-\mbox{\rm bight} I_{\Delta}$. Then $\mbox{\rm v}_{P_F}( I_{\Delta}) = m $ if and only if  there exists $W \subset F$ with $|W|=m$ such that $|W\cap G|\le m-1$ for any $G \in \mathcal{F}(\Delta)\setminus\{F\}$ and there exists $G_{W'}\in \mathcal{F}(\Delta)\setminus\{F\}$ such that $W'\subset G_{W'}$ for any $W' \subset F$ with $|W'|=m-1$.
\end{cor}
\begin{proof}
By Proposition \ref{v_P 1}, it suffices to consider the case $m=1$; in this case, it is clear. 
\end{proof}
\begin{cor}\label{}
Let $\Delta$ be a simplicial complex and let $F\in \mathcal{F}(\Delta)$. Then $\mbox{\rm v}_{P_F}( I_{\Delta})=1$ if and only if $F$ has a free vertex.
\end{cor}

Now we consider the height 2 case. First, we provide an upper bound for ${\rm v}_{P}$-numbers of Stanley--Reisner ideals in terms of the depth of the second symbolic power of Stanley--Reisner ideals of its Alexander dual $S/I_{\Delta^*}^{(2)}$. 
\begin{thm}\label{v_P depth}
Let $\Delta$ be a pure simplicial complex such that ${\rm ht} I_{\Delta}=2$. Then  
$$\mbox{\rm v}_P( I_{\Delta})\le n- \mbox{\rm depth}\hspace{0.05cm}S/I_{\Delta^*}^{(2)}-1\mbox{ for any }P\in \Ass S/I_{\Delta}$$ 
\end{thm}
\begin{proof}
We set $D=\mbox{\rm depth}\hspace{0.05cm}S/I_{\Delta^*}^{(2)} $ and take $P=(x_i,x_j)\in \Ass S/I_{\Delta}$.
If $D\ge 2 $, then 
the ${\rm diam}(\Delta^*) ^{(1)}\le2$ by \cite[Theorem 3.2]{RTY}.
Hence 
there exists   $ \{x_{i_1}\} \in V(\Delta^*)$ such that  $\{x_i, x_{i_1}\}, \{x_{i_1}, x_j\} \in \Delta^*$.
If $D\ge 3 $, then ${\rm diam}(\link_{ \Delta^*}\{x_{i_1}\}) ^{(1)}\le2$ by \cite[Theorem 3.2]{RTY}.
Hence there exists   $ \{x_{i_2}\} \in  V(\link_{ \Delta^*}\{x_{i_1}\}) ^{(1)})$ such that  $\{x_i, x_{i_1}\}, \{x_{i_1}, x_j\} \in \link_{ \Delta^*}\{x_{i_1}\}$.
We repeat this process and finally consider  the 1-skeleton $(\link_{ \Delta^*}\{x_{i_1}, x_{i_2}, \dots,  x_{i_{D-2}}\}) ^{(1)}=\link_{ \Delta^*}\{x_{i_1}, x_{i_2}, \dots,  x_{i_{D-2}}\}$.
Then there exists   $ \{x_{i_{D-1}}\} \in V(\link_{ \Delta^*}\{x_{i_1}, x_{i_2}, \dots,  x_{i_{D-2}}\})$ such that  $\{x_i, x_{i_{D-1}}\}, \{x_{i_{D-1}}, x_j\} \in \link_{ \Delta^*}\{x_{i_1}, x_{i_2}, \dots,  x_{i_{D-2}}\}$.
By considering $\{x_{i_1}, x_{i_2}, \dots,  x_{i_{D-2}}\} \in \Delta^*$, we get the inequality. 
\end{proof}

In particular, we provide a characterization of the ${\rm v}_{P}$-numbers when a simplicial complex satisfies a special condition. 
\begin{thm}\label{S2square}
Let $\Delta$ be a pure simplicial complex with ${\rm ht} I_{\Delta}=2$. If  $S/I_{\Delta^*}^{(2)}$ satisfies (S$_2$), then $$\mbox{\rm v}_P( I_{\Delta})= \indeg  I_{\Delta}-1\mbox{\quad for any }P\in \Ass S/I_{\Delta}$$
\end{thm}
\begin{proof}
Let $d^*=\dim S/I_{\Delta^*}$ and take $P=(x_i,x_j)\in \Ass S/I_{\Delta}$. Since $S/I_{\Delta^*}^{(2)}$ satisfies $(S_2)$, the 1-skeleton ${\rm diam}(\Delta^*) ^{(1)}\le2$ by \cite[Corollary 3.3]{RTY}. Hence 
there exists   $ \{x_{i_1}\} \in V(\Delta^*)$ such that  $\{x_i, x_{i_1}\}, \{x_{i_1}, x_j\} \in \Delta^*$.
Moreover, since $S/I_{\link_{ \Delta^*}\{x_{i_1}\}}^{(2)}$ also satisfies $(S_2)$, the 1-skeleton ${\rm diam}(\link_{ \Delta^*}\{x_{i_1}\}) ^{(1)}\le2$ by \cite[Corollary 3.3]{RTY}. Hence, there exists   $ \{x_{i_2}\} \in  V((\link_{ \Delta^*}\{x_{i_1}\}) ^{(1)})$ such that  $\{x_i, x_{i_1}\}, \{x_{i_1}, x_j\} \in \link_{ \Delta^*}\{x_{i_1}\}$. We repeat this process and finally consider $(\link_{ \Delta^*}\{x_{i_1}, x_{i_2}, \dots,  x_{i_{d^*-2}}\}) ^{(1)}$, that is, $\link_{ \Delta^*}\{x_{i_1}, x_{i_2}, \dots, x_{i_{d^*-2}}\}$. Then there is $ \{x_{i_{d^*-1}}\} \in V(\link_{ \Delta^*}\{x_{i_1}, x_{i_2}, \dots,  x_{i_{d^*-2}}\})$ such that $\{x_i, x_{i_{d^*-1}}\}, \{x_{i_{d^*-1}}, x_j\} \in \link_{ \Delta^*}\{x_{i_1}, x_{i_2}, \dots,  x_{i_{d^*-2}}\}$. By letting $F=\{x_{i_1}, x_{i_2}, \dots,  x_{i_{d^*-1}}\}$, we have  $|F|=n-\indeg  I_{\Delta}-1$ and $F\cup \{x_{i}\},  F\cup \{x_{j}\}\in   \Delta^*$, which leads to the conclusion by Theorem \ref{intersection}. 
\end{proof}
\begin{cor}\label{}
Let $\Delta$ be a pure simplicial complex with ${\rm ht} I_{\Delta}=2$. Then $\mbox{\rm v}_P( I_{\Delta})\le n- 3$ for any $P\in \Ass S/I_{\Delta}$ if and only if $\mbox{\rm depth} \hspace{0.05cm}S/I_{\Delta^*}^{(2)}\ge 2$.
\end{cor}
\begin{proof}
By Theorem \ref{v_P depth}, `` only if '' part is obvious. So, it suffices to show that ${\rm diam}(\Delta^*) ^{(1)}\le2$ by \cite[Theorem 3.2]{RTY}. We take $x,y\in V(\Delta^*)$ with $\{x, y\}\not\in \Delta^*$. Since $(x,y)\in  \Ass S/I_{\Delta}$ and $\mbox{\rm v}_{(x,y)}( I_{\Delta})\le n- 3$, by Theorem \ref{intersection}, we have $\max \{ |F_x\cap F_y | \,\,: F_x , F_y \in \Delta^{*}\mbox{ such that }  x\in F_x, y\in F_y \} \ge 1$, and hence there is $z\in V(\Delta^*)$ with $\{x, z\},   \{z, y\}\in \Delta^*$, that is, $\dist_{\Delta^*} (x,y) =2$.
\end{proof}
We need the following lemma to provide an upper bound for the ${\rm v}_{P}$-numbers in an arbitrary height. 
\begin{lemma}\label{nonvanish}
Let $\Delta$ be a  simplicial complex on the vertex set $[n]$.
If $ x_1, x_2,\dots, x_m \in [n]$ satisfy the following conditions;
\begin{enumerate}
\item $\partial (\{ x_1, x_2,\dots, x_m \}) \subset \Delta$, $\{ x_1, x_2,\dots, x_m \} \not\in \Delta$. 
\item There does not exist $z\in V(\Delta)$ such that $\{z, x_1,x_2,\dots, \hat{x_i} , \dots,  x_m\}\in \Delta$ for all $1\le i \le m$. 
\item For every $F\in \mathcal{F}(\Delta)$, there exists $i$ such that  $\{x_1,x_2,\dots, \hat{x_i} , \dots,  x_m\}\subset F$. 
\end{enumerate}
then $\widetilde{H}_{m-2}(\Delta;k)\ne 0$.
\end{lemma}
\begin{proof}
We fix an integer $m$ and we prove the lemma by  induction on the number of vertices (for $\ge m$). In the case that  $n=m$, there is nothing to prove. For the case that $n>m$, we take $z\in [n] \setminus \{ x_1, x_2,\dots, x_m \}$. By the conditions (2) and  (3), we may assume that $\{z, x_1,x_2,\dots, \hat{x_i} , \dots,  x_m\}\in \Delta$ for all $1\le i \le p$ and $\{z, x_1,x_2,\dots, \hat{x_i} , \dots,  x_m\} \not\in \Delta$ for all $p+1\le i \le m$ for some $p\le m-1$.By considering the Mayer--Vietoris sequence for $$0\rightarrow {\rm link}_{\Delta}\{z\}\rightarrow \Delta_{V(\Delta)\setminus\{z\}}\oplus{\rm star}_{\Delta}\{z\}\rightarrow\Delta\rightarrow 0,$$ we obtain that $\widetilde{H}_{m-2}(\Delta; k)\cong\widetilde{H}_{m-2}(\Delta_{V(\Delta)\setminus \{z\}}; k),$ where we notice that $\link_{\Delta}\{z\}$ is a cone with apex   $x_m$ and is contractible. Since $\Delta_{V(\Delta)\setminus \{z\}}$ also satisfies the conditions of the assumption, we have $\widetilde{H}_{m-2}(\Delta_{V(\Delta)\setminus \{z\}}; k) \ne 0$, by the induction hypothesis, the assertion follows. 
\end{proof}
Finally, we provide an upper bound for the ${\rm v}_{P}$-numbers in an arbitrary height in terms of $S_{{\rm ht}I_{\Delta}}$-{\rm depth} of $S/I_{\Delta^*}^{(2)}$. 
\begin{thm}\label{S$_h$-depth}
Let $\Delta$ be a pure simplicial complex such that ${\rm ht} I_{\Delta}=h\geq2$.
Then $$\mbox{\rm v}_P( I_{\Delta})\le n- \mbox{\rm S$_h$-depth}\hspace{0.05cm}S/I_{\Delta^*}^{(2)}-1
\mbox{ \quad for any }P\in \Ass S/I_{\Delta}.$$
\end{thm}
\begin{proof}
Let $D=  \mbox{\rm S$_h$-depth}\hspace{0.05cm}S/I_{\Delta^*}^{(2)} $ and take $P\in \Ass S/I_{\Delta}$. Without loss of generality, we may assume that $P=(x_{1},   x_{2},  \dots, x_{h})$. First we suppose that $ D\ge h $ and take $\boldmath{a}=(1,1,\dots, 1, 0,\dots, 0)$, which means that the first $h$ coordinates are 1, the others are 0.
Then by Corollary \ref{dual}, we have $\widetilde{H}_{h-2}(\Delta_{\boldmath{a}}(I_{\Delta^*}^{(2)}); k)=0$ and hence, by Lemma \ref{nonvanish}, there exists $z_1\in V(\Delta_{\boldmath{a}}(I_{\Delta^*}^{(2)}))$ such that $\{z_1, x_1,x_2,\dots, \hat{x_i} , \dots,  x_h\}\in \Delta_{\boldmath{a}}(I_{\Delta^*}^{(2)})$ for all $1\le i \le h$. If $D\ge h+1$, then, by Corollary \ref{dual}, we have $\widetilde{H}_{h-2}(\Delta_{\boldmath{a}}(I^{(2)}_{\link_{\Delta^*}}\{z_1\}); k)=0$ and hence there exists $z_2\in V(\Delta_{\boldmath{a}}(I_{\Delta^*}^{(2)}))$ such that $\{z_1, z_2, x_1,x_2,\dots, \hat{x_i} , \dots,  x_h\}\in \Delta_{\boldmath{a}}(I_{\Delta^*}^{(2)})$ for all $1\le i \le h$ by Lemma \ref{nonvanish}. 
Also, if $D\ge h+2$, then again by Corollary \ref{dual}, we have $\widetilde{H}_{h-2}(\Delta_a(I^{(2)}_{\link_{\Delta^*}}\{z_1, z_2\}); k)=0$ and hence there exists $z_3\in V(\Delta_{\boldmath{a}}(I_{\Delta^*}^{(2)}))$ such that $\{z_1, z_2, , z_3, x_1,\dots, \hat{x_i} , \dots,  x_h\}\in \Delta_a(I_{\Delta^*}^{(2)})$ for all  $1\le i \le h$ by Lemma \ref{nonvanish}. We repeat this process and finally we consider $\link_{\Delta_{\boldmath{a}}(I^{(2)})}\{z_1, z_2, \dots, z_{D-h}\}$. Since $ K_{S/I_{\Delta^*}^{(2)}} ^{D}=0$, we have $\widetilde{H}_{h-2}(\Delta_{\boldmath{a}}(I^{(2)}_{\link_{\Delta^*}}\{z_1, z_2, \dots, z_{D-h}\}); k)=0$. By Lemma \ref{nonvanish}, there is $z_{D-h+1} \in V(\Delta_{\boldmath{a}}(I^{(2)}_{\link_{\Delta^*}}\{z_1, z_2,  \dots, z_{D-h}\})$.
Therefore we obtain that $\mbox{\rm v}_P( I_{\Delta})\le n-h-(D-h+1)= n-  \mbox{\rm S$_h$-depth}\hspace{0.05cm}S/I_{\Delta^*}^{(2)}-1$. 
\end{proof}

%%%%%%%%%%%%%%%%%%%%%%%%%%%%%%%%%%%%%%%%%%%%%%%%%%%%%%%%%%%%%%%%%%%%%%
   
 \section{the {\rm v}-numbers of Stanley--Reisner ideals of 2-pure complexes} 
In this section, we give a necessary and sufficient condition for when the {\rm v}-number and the dimension are equal for Stanley--Reisner ideals. For the purpose, we prepare the next lemma about 2-pure simplicial complex.  
\begin{lemma}\label{no-boundary}
Let $\Delta$ be a $(d-1)$-dimensional pure simplicial complex on the vertex set $V$. Then  $\Delta$ is 2-pure  if and only if  $\tilde{H}_0(\link_{\Delta}(F); k) \ne 0$  for any $F\in \Delta$ with $\dim F=d-2$.
\end{lemma}
\begin{proof} 
We take $G\in \Delta$ with $\dim G=d-2$. Then there exists a facet,  say $F$, of  $\Delta$ with  $G\subset F$. By setting $F=G\cup \{x\}$, there exists a facet,  say $F'$, of  $\Delta _{V\setminus \{x\}}$ with  $G\subset F'$. Since $F'$ is also a facet of  $\Delta$, we have  $\tilde{H}_0(\link_{\Delta}(G); k) \ne 0$. To prove the converse, it is enough to show that for any  $x\in V$ and face $H \in \Delta _{V\setminus \{x\}}$,there exists $F \in\mathcal{F}(\Delta _{V\setminus \{x\}})$ such that $H\subset F$ and $\dim F=d-1$. Now, since $\Delta$ is pure, there exists $F' \in\mathcal{F}(\Delta)$ with  $H\subset F'$. We may assume that $x \in F^{\prime}$, otherwise the proof is done. By the assumption $\widetilde{H}_{0}({\rm link}_{\Delta}(F^{\prime}\setminus\{ x\}))\neq 0$, there exists a facet, say $F$, of $\Delta$ with $F \ne F'$ such that $x\notin F$. Then $F\in\mathcal{F}(\Delta _{V\setminus \{x\}})$ with  $H\subset F$, the assertion follows. 
\end{proof}
We show  the main theorem in this section, which is a generalization of  \cite[Theorem 4.5]{JV} on the edge ideal case.
\begin{thm}\label{w2}
Let $\Delta$ be a $(d-1)$-dimensional  simplicial complex on the vertex set $V=\{x_1, \dots, x_n\}$. Then   $\mbox{\rm v}( I_{\Delta})= d$ if and only if  $\Delta$ is 2-pure.
\end{thm}
\begin{proof} 
Take the clutter $ \mathcal{C}$ such that $I( \mathcal{C})=I_{\Delta}$. By the above lemma, it is enough to show that $\mbox{\rm v}( I_{\Delta})= d$ if and only if $\tilde{H}_0(\link_{\Delta}(F); k) \ne 0$ for any $F\in \Delta$ with $\dim F=d-2$. First, we assume that the condition that $\tilde{H}_0(\link_{\Delta}(F); k) \ne 0$  for any $F\in \Delta$ with $\dim F=d-2$. From the fact that $\mbox{\rm v}( I_{\Delta})= d$ if and only if $\Delta$ is pure and $\mathcal{A}_{ \mathcal{C}}=\mathcal{F}(\Delta)$, by \cite[Corollary 3.7]{JV}, it suffices to show that $\mathcal{A}_{\mathcal{C}}=\mathcal{F}(\Delta)$. Let $F \in \Delta\setminus \mathcal{F}(\Delta)$, then there exists a face $G$ of $\Delta$ with $\dim G=d-2$ such that $F \subset G$. Also, by the assumption, there exist two facets $F_1, F_2 \in \mathcal{F}(\Delta)$ such that $G = F_1 \cap F_2$. Then we can write $F_1=G  \cup \{x_i\}$ and  $F_2=G  \cup \{x_j\}$. Now, since $G  \cup \{x_i, x_j\}\not\in \Delta$, there exists a minimal non-face $H$ such that $\{x_i, x_j\}\subset H \subset G  \cup \{x_i, x_j\}$. Since $H \cap N_{\mathcal{C}}(G)=\emptyset$, $N_{\mathcal{C}}(G)$ is not a vertex cover of $\mathcal{C}$ and hence $G\not\in \mathcal{A}_{\mathcal{C}}$, which leads to $\mathcal{A}_{\mathcal{C}}=\mathcal{F}(\Delta)$. Next we show the other implication by contradiction. Suppose that there exists a face $G$ of $\Delta$ with ${\rm dim}G=d-2$ such that $G$ is contained in a unique facet, say $F$. Now, it is easy to see that $N_{\mathcal{C}}(G)=V\setminus F$ is a minimal vertex cover of $\mathcal{C}$, that is, $G \in  \mathcal{A}_{\mathcal{C}}$. This means that $ \mathcal{A}_{\mathcal{C}} \ne \mathcal{F}(\Delta)$. then, by \cite[Corollary 3.7]{JV}, we have $\mbox{\rm v}( I_{\Delta})< d$.
\end{proof}
By the definitions of 2-pure and 2-Cohen--Macaulay and the fact that both Gorenstein* and matroid* complexes are 2-Cohen--Macaulay, the following corollary easily holds. 
\begin{cor}\label{2-CM}
If $\Delta$ is 2-Cohen-Macaulay, Gorenstein or a matroid, then $$\mbox{\rm v}( I_{\Delta})=\reg k[\Delta].$$
\end{cor}
\begin{proof} 
If $\Delta$ is 2-Cohen--Macaulay, then, by Theorem \ref{w2}, the equality holds. Hence, the equality also holds for a Gorenstein*  complex and a matroid* complex. Since taking a cone does not change  the values of  the v-number  and the  regularity, and also a Gorenstein complex (resp., a matroid complex) is an iterated cone of a  Gorenstein* complex (resp., a matroid complex), the assertion follows.
\end{proof}

%%%%%%%%%%%%%%%%%%%%%%%%%%%%%%%%%%%%%%%%%%%%%%%%%%%%%%%%%%%%%%%%%%%%%%

\section{Range of v-numbers of  monomial ideals of pure height two } 
In this section, we construct examples of a pure height two simplicial complex $\Delta$ such that $({\rm reg}I_{\Delta}, {\rm v}(I_{\Delta}), {\rm indeg}\hspace{0.05cm}I_{\Delta})$ holds for given integers,
which gives a refined  version of \cite[Theorem  3.10]{S}.  Also, we give counter-examples for several questions. We show the main theorem in this section. 
\begin{thm}\label{range}
For $p\ge q \ge r \ge 1$, there exists  a pure simplicial complex $\Delta$ with ${\rm ht}I_{\Delta}=2$ such that $\reg  I_{\Delta}=p+1$,  $\mbox{\rm v}( I_{\Delta})=q$, $\indeg I_{\Delta}=r+1$.
\end{thm}
\begin{proof} 
We set $F_1=\{ x_1, x_2, \dots , x_{p-r+1}\}$, $F_2=\{x_1, x_2, \dots , x_{p-q},  x_{p-r+2}\}$ and
$$\Gamma=\langle F_1, F_2, \{x_{p-r+3}\}, \{x_{p-r+4}\}, \dots , \{x_{p+2}\}\rangle.$$
Notice that $$\mathcal{G}(I_{\Gamma})=\left(\bigcup_{i=r}^{q}\{x_{p-i+1}x_{p-r+2}\}\right)\cup\left(\bigcup_{k=3}^{r+2}\left(\bigcup_{j\neq p-r+k}\{x_{j}x_{p-r+k}\}\right)\right),$$ which implies $\Gamma^{*}$ is pure with ${\rm ht}I_{\Gamma^{*}}=2$. Now, we have $\pd k[\Gamma]=p+1$ and  ${\rm ht} I_{\Gamma}=(p+2)-(p-r+1)=r+1$. If $p=q$, then each intersection of two facets is empty, and hence, by Theorem \ref{intersection}, ${\rm v}(I_{\Gamma^{*}})=p=q$. So, we may assume that $p\neq q$. Now, since $x_{p-r+1}\in F_{1}$, $x_{p-r+2}\in F_{2}$ and $|F_{1}\cap F_{2}|=p-q\geq 1$, again by Theorem \ref{intersection}, $\mbox{\rm v}_{(x_{p-r+1},x_{p-r+2})}( I_{\Gamma^{*}})=(p+2)-(p-q)-2= q.$ Hence  $\reg  I_{\Gamma^*}=\pd k[\Gamma]=p+1$ (\cite[Corollary 0.3]{T99}),  $\mbox{\rm v}( I_{\Gamma^*})=q$, $\indeg I_{\Gamma^*}=r+1$. By taking $\Delta=\Gamma^{*}$, the assertion follows. 
\end{proof}
In the rest of this section, we give answers for several questions. 
\begin{quest}\cite[Question 3.12]{S}
Does there exist a graph $G$ such that $G$ is not a complete multipartite graph, but $\mbox{\rm v}( J(G))> \mbox{\rm bight}I(G)-1?$
\end{quest}
We prove that this question is true. 
\begin{exam} For $m>>\ell$, we set 
\begin{enumerate}[ ]
\item $F_1=\{ x_1, x_2, \dots , x_{m}\}$, 
\item $F_2=\{x_1, x_2, \dots , x_{\ell},  x_{m+1},  \dots, x_{2m-\ell}\}$, 
\item  $F_3=\{ x_{2m-\ell+1},  x_{2m-\ell+2}, \dots ,  x_{3m-\ell} \}.$ 
\end{enumerate}
Let $\Delta=\langle F_1, F_2, F_3 \rangle$, which is a pure flag complex. Then, 
$\mbox{\rm v}( I_{\Delta^*})=(3m-\ell)-\ell-2=3m-2\ell-2$, and  ${\rm bight}I_{\Delta}={\rm ht} I_{\Delta}=(3m-\ell)-m=2m-\ell$, and hence 
$${\rm v}(I_{\Delta^{*}})=3m-2\ell-2> 2m-\ell-1={\rm bight}I_{\Delta}-1,$$
if $m>\ell+1$. 
\end{exam}

\begin{quest}\cite[Question 5.5]{SS}
For a squarefree monomial ideal $I$, does $\mbox{\rm v}(I)\le \depth S/I $ hold?  
Also can we say that   $\mbox{\rm v}(I)\ge \dim S/I-\depth S/I $?
\end{quest}
We give negative answers for these questions. 
\begin{exam} 
Let $\Delta$ be the 6-vertex triangulation of real projective plane (See Example \ref{RealP}).
Now we consider its Stanley--Reisner ring $k[\Delta]$ such that the base filed $k$ has 
$\chara k=2$.
Then $\mbox{\rm v}( I_{\Delta})=3$ and    $\depth k[\Delta]=2$.                        
\end{exam}
\begin{exam} 
Set $F_1=\{ x_1, x_2,  x_{3}\}$, $F_2=\{x_4\}$ and  $\Delta=\langle F_1, F_2 \rangle$.
Then  $\mbox{\rm v}( I_{\Delta})=1$ (for the proof see  \cite[Proposition 3.19]{JV}) and $\dim  k[\Delta]-\depth  k[\Delta] =3-1=2$.
\end{exam}

%%%%%%%%%%%%%%%%%%%%%%%%%%%%%%%%%%%%%%%%%%%%%%%%%%%%%%%%%%%%%%%%%%%%%%

\section{the v-numbers of  edge ideals of \\ very well-covered graphs and multi-whisker graphs} 
In the last section of this paper, we give the {\rm v}-number of the edge ideal of very well-covered graphs and multi-whisker graphs. First, we reprove the {\rm v}-number of the edge ideal of very well-covered graphs by using Theorem \ref{structure}. 
\begin{thm}\cite[Theorem 4.3]{GRV}\label{very well-covered edge}
If $G$ is  a very well-covered  graph, then we have $\mbox{\rm v}( I(G)) \le \reg S/I(G)$.
\end{thm}
\begin{proof} 
First we prove the inequality holds for a Cohen--Macaulay  very well-covered graph $H$ with $V(H)=X_0\cup Y_0$  where  $X_0=\{x_1, \dots, x_{h_0}\}$ and  $Y_0=\{y_1, \dots, y_{h_0}\}$ such that $X_0$ is a minimal vertex cover of $H$ and  $Y_0$ is a maximal independent set of $H$ and $e_1=\{ x_1, y_1\}, \dots ,e_{h_0}= \{  x_{h_0},  y_{h_0}\} \in E(H)$. By \cite[Lemma 3.5]{CRT} we may assume that $i\le j$, if $\{x_i, y_j\}\in E(H)$. First, choose $e_{1}$, then choose the minimal $i_2$ such that  $\{e_1, e_{i_2}\}$ is an induced matching in $H$, and next choose the minimal $i_3$ such that  $\{e_1, e_{i_2}, e_{i_3}\}$ is an induced matching in $H$. We repeat this process as much as possible and  finally we assume to get  $\{e_1, e_{i_2}, \dots, e_{i_r}\}$. Then we have $r\le \reg S/I(H)$ by \cite[Lemma 2.2]{Ka}. On the other hand, since $N _H(\{x_1, x_{i_2}, \dots, x_{i_r}\})$ is a vertex cover of $H$ by the condition that $i\le j$ if $\{x_i, y_j\}\in E(H)$, we have $\{x_1, x_{i_2}, \dots, x_{i_r}\} \in \mathcal{A}_H$ by\cite[Lemma 3.4 (b)]{JV}. Therefore, by \cite[Theorem 3.5]{JV}, $\mbox{\rm v}( I(H)) \le r \le \reg S/I(H)$. Next we consider the general case. By Theorem \ref{structure}, a very well-covered graph $G$ can  be expressed as  $G=H(n_1, n_2, \dots , n_{h_0})$  for some $n_1, n_2, \dots , n_{h_0}\ge 1$. Now, since $ A=\{z_{i_1j_1}, \dots, z_{i_pj_p}\}\in \mathcal{A}_{G}$ if and only if $A'=\{z_{i_1}, \dots, z_{i_p}\}\in \mathcal{A}_{H}$, we have $\min \{|A| \,\,: A \in \mathcal{A}_{G}\}=\min \{|A| \,\,:A \in \mathcal{A}_{H}\}$, Hence by \cite[Theorem 3.5]{JV}, we have $$\mbox{\rm v}( I(G))=\mbox{\rm v}( I(H))\le\reg S/I(H)=\reg S/I(G),$$ which completes the proof. 
\end{proof}

Finally, we generalize \cite[Theorem 3.20]{JV}, which treats whisker graphs, to multi-whisker graphs. To this end, we prepare the notation about multi-whisker graphs. For a subset $B$ of $Y_{[h]}$, if $B$ contains $y_{i, j}$, then we replace $y_{i, j}$ with $y_{i, 1}$. Let $B^{1}$ denote the set obtained by applying this operation to any elements of $B$.
\begin{thm}
Let $G=G_0[n_1, n_2, \dots , n_{h}]$ be the multi-whisker graph associated with a graph $G_{0}$. Then we have :
\begin{enumerate}
\item ${\rm v}(I(G))=i(G_{0})\leq i(G)={\rm depth}\hspace{0.05cm}S/I(G)$. 
\item $\mbox{\rm v}( I(G)) \le \reg S/I(G)$.
\end{enumerate}
\end{thm}
\begin{proof}
(1)${\rm v}(I(G))\leq i(G_{0})$ follows from the proof in  {\cite[Theorem 3.20]{JV}}. Hence it suffices to show that ${\rm v}(I(G))\geq i(G_{0})$. Let $A, B$ be independent sets of $G$ with $A\subset X_{[h]}$, $B\subset Y_{[h]}$ and $N_{G}(A\cup B)$ is a minimal vertex cover of $G$. 
Then, $B^{1}$ is an independent set of $G$, $N_{G}(A\cup B)=N_{G}(A\cup B^{1})$ and ${\rm deg}(x_{A}y_{B})\geq{\rm deg}(x_{A}y_{B^{1}})$. Therefore we can consider the case with respect to $B^{1}$. By replacing the whisker in $B^{1}$ with the corresponding vertex in $G_{0}$ in the same way as in the proof {\cite[Theorem 3.20]{JV}}, we obtain the first equality. Also, from \cite[Corollary 4.3]{MPT}, the assertion follows.    

(2) By (1), we have ${\rm v}(I(G))=i(G_{0})={\rm v}(I(G_0[1,1,\dots, 1]))$.
By \cite[Theorem 1.3]{MMCRTY} and \cite[Corollary 4.4]{MPT} we know that $\reg S/I(G)=\mbox{im}(G)=\mbox{im}(G_0[1,1,\dots, 1])=\reg S/I(G_0[1,1,\dots, 1])$. 
Hence we have the desired result by  \cite[Theorem 3.20]{JV}.
\end{proof}

\begin{acknowledgement}
We are grateful to the anonymous referee for their careful reading and valuable comments. The research of Yuji Muta was partially supported by ohmoto-ikueikai and JST SPRING Japan Grant Number JPMJSP2126. The research of Naoki Terai was partially supported by JSPS Grant-in Aid for Scientific Research (C) 24K06670.
\end{acknowledgement}
 
%%%%%%%%%%%%%%%%%%%%%%%%%%%%%%%%%%%%%%%%%%%%%%%%%%%%%%%%%%%%%%%%%%%%%%
\bibliographystyle{amsplain}

\end{document}